\documentclass{amsart}
\usepackage[a4paper, margin=2.8cm]{geometry}
\usepackage{graphicx} % Required for inserting images
\usepackage[utf8]{inputenc}
\usepackage[T1]{fontenc}
\usepackage{amsthm}
\usepackage{thmtools}
\usepackage{amsmath}
\usepackage{amsfonts}
\usepackage{hyperref}
\usepackage[capitalize]{cleveref}
\usepackage{amssymb}
\usepackage{xcolor}
\usepackage{yfonts}
\usepackage{tkz-fct}
\usepackage{booktabs}
\usepackage{setspace}
\usepackage{tikz-cd}
\usetikzlibrary{cd}
\usetikzlibrary{arrows}
\usepackage{quiver}
\DeclareMathOperator{\ach}{\mathrm{CH}}
\DeclareMathOperator{\ch}{\underline{\mathrm{CH}}}
\DeclareMathOperator{\ot}{\otimes}
\newtheorem{theorem}{Theorem}[section]
\newtheorem{prop}[theorem]{Proposition}
\newtheorem{lemma}[theorem]{Lemma}

\newtheorem{definition}[theorem]{Definition}
\newtheorem{remark}[theorem]{Remark}
\newtheorem{corollary}[theorem]{Corollary}

\hypersetup{
    colorlinks=true,
    linkcolor=blue,
    filecolor=blue,      
    urlcolor=blue,
    citecolor=violet
    }

\title{Decompositions of Chow rings of direct sums of matroids}
\author{Paweł Pielasa}
\address{University of Cambridge, Department of Pure Mathematics and Mathematical Statistics}
\email{pp554@cam.ac.uk}
\date{}
\begin{document}

\begin{abstract}
    We prove two dual recursive decompositions as a graded $\ch(M) \otimes \ch(N)$-module of the Chow ring $\ch(M\oplus N)$ of the direct sum of matroids. We use this to obtain a decomposition of $\ch(M\oplus N)$ into irreducible $\ch(M) \otimes \ch(N)$-modules. The result implies a new recursive formula for the Eulerian numbers. Similarly, we find a recursive decomposition of the augmented Chow ring $\ach(M \oplus N)$ into $\ach(M) \otimes \ach(N)$-modules, generalizing some of the results of \cite{ssd}. We prove analogous decompositions of (augmented) Chow polynomials of weakly ranked posets in the sense of \cite{CF25}.
\end{abstract}

\maketitle

\section{Introduction}

The direct sum is one of the most fundamental operations that can be applied to matroids, assigning to a pair of matroids $M,N$ on ground sets $[m],[n]$ a matroid $M\oplus N$ on the ground set $[m] \sqcup [n]$ whose flats are given by $F \sqcup G$ for pairs of flats $F,G$ in $M,N$. Denote by $\ch(M)$, $\ch(N)$ the Chow rings of $M$, $N$ and by $\ach(M)$, $\ach(N)$ the augmented Chow ring of $M$, $N$. There is a structure of a $\ch(M) \otimes \ch(N)$-module on $\ch(M \oplus N)$ and a structure of an $\ach(M) \otimes \ach(N)$-module on $\ach(M \oplus N)$.\\

In general, proving statements about (augmented) Chow rings, or obtaining an explicit understanding of any specific Chow ring, can be a challenging task. Such investigations often lead to deep combinatorial insights. The most common approach to studying Chow rings of matroids is by induction, establishing the desired properties for “smaller” matroids obtained through suitable matroid operations (see for example \cite{AHKcomb}). In \cite{ssd} the authors studied the decomposition of the Chow ring $\ch(M)$ of a loopless matroid $M$ as a $\ch(M\backslash i)$-module, where $M\backslash i$ denotes the contraction of $M$ by $i\in [m]$. They obtained its decomposition into irreducible $\ch(M\backslash i)$-modules, as well as an analogous decomposition of $\ach(M)$ into irreducible $\ach(M \backslash i)$-modules. We state it below in the case when $i$ is a coloop of $M$.
\begin{prop}\label{ssddlachow}
    \begin{enumerate}
        \item If $i$ is a coloop of $M$ then there is an isomorphism of graded $\ch(M\backslash i)$-modules
        \[ \ch(M) \simeq \ch(M\backslash i) \oplus \ch(M \backslash i)[-1] \oplus \underset{\emptyset \subsetneq F \subsetneq M \backslash i}{\bigoplus} \ch(M_{F \cup i}) \otimes \ch(M^F) [-1], \]
        where the direct sum runs over all nonempty proper flats of $M\backslash i$.
        \item If $i$ is a coloop of $M$ then there is an isomorphism of graded $\ach(M\backslash i)$-modules
        \[ \ach(M) \simeq \ach(M \backslash i) \oplus \ach(M \backslash i)[-1] \oplus \underset{\emptyset \subsetneq F \subsetneq M\backslash i}{\bigoplus} \ach(M^F) \otimes \ch(M_{F \cup i}) [-1], \]
        where the direct sum runs over all proper flats of $M\backslash i$.
        %\item If $i$ is not a coloop of $M$ then there is an isomorphism of graded $\ch(M \backslash i)$-modules
        %\[ \ch(M) \simeq \ch(M \backslash i) \oplus \underset{\emptyset \subsetneq F \subsetneq M \backslash i}{\bigoplus} \ch(M_{F \cup i}) \otimes \ch(M^F) [-1], \]
        %where the direct sum runs over all nonempty proper flats of $M \backslash i$ such that $F$ and $F \cup i$ are flats of $M$.
    \end{enumerate}
\end{prop}
%\begin{prop}\label{ssddlaaug}
%    \begin{enumerate}
%        \item If $i$ is a coloop of $M$ then there is an isomorphism of graded $\ach(M\backslash i)$-modules
%        \[ \ach(M) \simeq \ach(M \backslash i) \oplus \ach(M \backslash i)[-1] \oplus \underset{\emptyset \subsetneq F \subsetneq M\backslash i}{\bigoplus} \ach(M^F) \otimes \ch(M_{F \cup i}) [-1], \]
%        where the direct sum runs over all proper flats of $M\backslash i$.
%        \item If $i$ is not a coloop of $M$ then there is an isomorphism of graded $\ach(M \backslash i)$-modules
%        \[ \ach(M) \simeq \ach(M \backslash i) \oplus \underset{\emptyset \subsetneq F \subsetneq M \backslash i}{\bigoplus} \ch(M_{F \cup i}) \otimes \ch(M^F) [-1], \]
%        where the direct sum runs over all proper flats of $M \backslash i$ such that $F$ and $F \cup i$ are flats of $M$.
%    \end{enumerate}
%\end{prop}
These results were then crucially used in \cite{singular} to prove the Kähler package of the intersection cohomology module of a matroid. This implied remarkable combinatorial results, including the Dowling and Wilson's conjecture and  the non-negativity of the coefficients of the Kazdhan-Lusztig polynomials of matroids (introduced in \cite{KLpolyofmatroid}).\\

Motivated by the philosophy above we want to understand the decomposition of the (augmented) Chow ring of a direct sum of matroids into simpler -- or even irreducible -- $\ch(M) \otimes \ch(N)$-modules (respectively into $\ach(M)\otimes \ach(N)$-modules). In the following we assume all matroids to be loopless.

\subsection{Main results}
Let us briefly summarize the main results of our paper.
\begin{restatable}{theorem}{maithm} \label{thm1}
    There is an isomorphism of $\ch(M) \ot \ch(N)$-modules
    \begin{flalign*}
        \ch(M\oplus N) \simeq \ch(M)\ot \ch(N) \oplus
         \underset{F,G}{\bigoplus} \ch(M_F \oplus N_G) \ot \ch(M^F) \ot \ch(N^G) [-1],
    \end{flalign*}
    where we take the direct sum over all pairs nonempty flats $F,G$ in $M,N$.
\end{restatable}
We give a sketch of the geometric interpretation of the above theorem in the realizable case in Section \ref{sectionposet}.
We can also formulate a dual version of this decomposition.
\begin{restatable}{theorem}{maithmdwa} \label{thm2}
    There is an isomorphism of $\ch(M) \ot \ch(N)$-modules
    \begin{align*}
        \ch(M\oplus N) \simeq \ch(M)\ot \ch(N) \ \oplus
         \underset{F,G}{\bigoplus} \ch(M^F \oplus N^G) \ot \ch(M_F) \ot \ch(N_G) [-1],
    \end{align*}
    where we take the direct sum over all pairs flats $F\subsetneq M$, $G \subsetneq N$.
\end{restatable}
\begin{remark}
    Assuming that $N$ is a matroid with a ground set consisting of one element we recover the first part of Proposition \ref{ssddlachow}. In that case the decompositions in Theorem \ref{thm1} and Theorem \ref{thm2} agree, however this is not the case in general.
\end{remark}
The most fundamental example of a matroid is the boolean matroid on the ground set $[n]$, that is a matroid for which all subsets of $[n]$ are flats. We denote the Hilbert series of the boolean matroid by
\begin{align*}
    A_n(x) = \underset{k=0}{\sum^n} A(n,k)x^k.
\end{align*}
It is the Hilbert function of the permutohedral variety and the coefficients $A(n,k)$ are called the \textit{Eulerian numbers}. Note that the direct sum of boolean matroids with ground sets $[n]$ and $[m]$ is a boolean matroid on the set $[n+m]$. Thus we can apply the decomposition above to a pair of boolean matroids to get the following new recursive formula for the Eulerian numbers:
\begin{corollary}\label{coreuler}
    The Eulerian polynomials satisfy the following relation:
    \begin{align*}
        A_{n+m}(x) = A_n(x) \cdot A_m(x) + x\cdot\underset{0<r \leq n, 0<s \leq m}{\sum}\binom{n}{r} \binom{m}{s}  A_r(x) A_s(x) A_{n+m-r-s}(x).
    \end{align*}
\end{corollary}
In the case $m=1$, the formula above specializes to \cite[Theorem 1.5]{Eulerianbook}.

Applying Theorem \ref{thm1} recursively leads to a decomposition of $\ch(M\oplus N)$ into irreducible $\ch(M) \ot \ch(N)$-modules, generalising the results of \cite{ssd}.

\begin{restatable}{corollary}{corrr}\label{corirred}
    The Chow ring $\ch(M \oplus N)$ decomposes into the direct sum
    \begin{align*}
    \underset{\substack{ \emptyset \subsetneq F_1 \subsetneq \ldots \subsetneq F_k \subseteq M \\ \emptyset \subsetneq G_1 \subsetneq \ldots \subsetneq G_k \subseteq N}}{\bigoplus} \ch(M^{F_1}) \otimes \ch(M_{F_1}^{F_2}) \otimes \ldots \otimes \ch(M_{F_k})
    \otimes \ch(N^{G_1}) \otimes \ch(N_{G_1}^{G_2}) \otimes \ldots \otimes \ch(N_{G_k})[-k]
    %\underset{\substack{ \emptyset \subsetneq F_1 \subsetneq \ldots \subsetneq F_k \subsetneq M \\ \emptyset \subsetneq G_1 \subsetneq \ldots \subsetneq G_k \subsetneq N}}{\bigoplus} \ch(M^{F_1}) \otimes \ch(M_{F_1}^{F_2}) \otimes \ldots \otimes \ch(M_{F_k}) \otimes \\
    %\otimes \ch(N^{G_1}) \otimes \ch(N_{G_1}^{G_2}) \otimes \ldots \otimes \ch(N_{G_k})[-k]
    \end{align*}
    of $\ch(M) \otimes \ch(N)$-modules, where we treat the length-$0$ flag as $\ch(M) \otimes \ch(N)$.
    %the first direct sum runs over flags of flats such that $F_k=M$, or $G_k=N$.
\end{restatable}
The irreducibility of each factor is clear, as they are quotients of $\ch(M) \otimes \ch(N)$.\\

We also give a decomposition of the augmented Chow ring $\ach(M\oplus N)$ of a direct sum of matroids as a direct sum of $\ach(M) \otimes \ach(N)$-modules.

\begin{restatable}{theorem}{thmaugg}\label{thmaug}
    There is an isomorphism:
    \[ \ach(M\oplus N) \simeq \ach(M) \otimes \ach(N) \oplus \underset{F\subsetneq M, G \subsetneq N}{\bigoplus} \ach(M^F \oplus N^G) \otimes \ch(M_F)\otimes \ch(N_G)  [-1] \]
    of $\ach(M) \otimes \ach(N)$-modules, where the direct sum runs over all pair of proper flats $F\subsetneq M$, $G \subsetneq N$.
\end{restatable}

Similarly, this decomposition can be used to find the decomposition of $\ach(M\oplus N)$ into irreducible $\ach(M) \otimes \ach(N)$-modules.
\begin{remark}
    In the case that $N$ is a matroid with one-element ground set $\{i\}$ we recover a result refining the decomposition in the second part of Proposition \ref{ssddlachow} by additionally taking into account the action of the factor $\mathbb{Q}[y_i]/(y_i^2)$ in the tensor product $\ach(M) \otimes \ach(N)$.
\end{remark}

A key ingredient in the proofs of Theorem \ref{thm1}, Theorem \ref{thm2}, and Theorem \ref{thmaug} are the analogous equalities of Chow functions of weakly ranked posets introduced in \cite{CF25}, see Section \ref{subsecposet} for a summary of basic notions in KLS-theory.

\begin{restatable}{theorem}{chowwfun}\label{thmchowfun}
    Let $H_1, H_2$, and $H$ be the Chow functions of $\mathcal{P}_1$, $\mathcal{P}_2$, and $\mathcal{P}_1 \times \mathcal{P}_2$ with respect to $\kappa_1$, $\kappa_2$, and $\kappa$ respectively, viewed as elements of $\mathrm{Mat_{\mathcal{P}\times\mathcal{P}}}(\mathbb{Z}[x])$. Then the following relations hold:
    \begin{align}\label{rowH1}
        H = H_1 \otimes H_2 + x H\cdot (H_1 \otimes H_2) - xH\cdot (I \otimes H_2) - xH\cdot (H_1 \otimes I) + xH,
    \end{align}
    and
    \begin{align}\label{rowH2}
        H = H_1 \otimes H_2 + x(H_1 \otimes H_2)\cdot H - x(I \otimes H_2)\cdot H - x(H_1 \otimes I)\cdot H + xH.
    \end{align}
\end{restatable}

We also prove a similar result for the augmented Chow functions of weakly ranked posets.

\subsection{Acknowledgments}

I would like to express my deepest gratitude to Botong Wang for suggesting the problem. I would like to thank the Institute of Advanced Study for a great working environment. I thank Louis Ferroni, June Huh, and Mateusz Michałek for many helpful discussions. The author's visit as part of the Special Year on Algebraic and Geometric Combinatorics was funded as part of the action I.1.5. of the IDUB \textit{Excellence Initiative Research Univeristy} program at the University of Warsaw.

\section{Preliminaries}

\subsection{Matroids and their Chow rings} \label{subsectionmatroids}

Ever since their introduction in \cite{originalmatroid}, as a combinatorial generalization of an arrangement of points in a projective space, matroids have played an important role in combinatorics and geometry. The datum of a matroid $M$ consists of a ground set $[m]$ and a set of subsets $F \subseteq [m]$ called flats, satisfying the following axioms:
\begin{itemize}
    \item The whole ground set $[m]$ is a flat.
    \item The intersection $F_1 \cap F_2$ of any two flats is a flat. 
    \item For a flat $F$ and an element $a \in [m]\backslash F$ there exists a unique flat $F'$ such that $F \cup a \subseteq F'$, and if $F \cup a \subseteq F''$ for any other flat $F''$, then $F' \subset F''$.
\end{itemize}
Matroids can be equivalently characterized in terms of independent sets, basis, or circuits. For a thorough introduction to the theory of matroids see \cite{Oxley}. Let $V$ be a $\Bbbk$-vector space. The axioms above are in particular satisfied if we are given a configuration of points $\{ p_i \}_{i \in [m]}$ in the projective space $\mathbb{P}(V)$ and we define a flat to be a maximal subset of points lying on some fixed linear subspace of $\mathbb{P}(V)$. In that case we say that the matroid is realizable over $\Bbbk$. An element $i\in M$ is called a coloop of $M$ if and only if the subset $M\backslash i$ is a flat. In this paper we only consider loopless matroids, that is, we always assume that the empty set is a flat.
We define the following matroid operations:
\begin{definition}
    \begin{enumerate}
        \item Let $S$ be a subset of $M$. Then the matroid $M^S$, called the \textit{restriction} of $M$ to $S$, is defined as a matroid with a ground set $S$, whose flats are flats of $M$ contained in $S$. For $S = M \backslash i$ we denote the restriction $M^S$ by $M \backslash i$ and call the operation the \textit{deletion} of $i$ in $M$.
        \item Let $F$ be a flat of $M$. Then the matroid $M_F$, called the \textit{contraction} of $M$ with respect to $F$, is defined as a matroid with a ground set $M\backslash F$, whose flats are given by $G\backslash F$ for flats $G$ of $M$ containing $F$.
    \end{enumerate}
\end{definition}
In the following we denote by $d$ be the rank of the matroid $M$, that is the length of its maximal flag of proper flats. In \cite{Chowringintroduced} the authors assigned to a matroid $M$ a graded commutative $\mathbb{Q}$-algebra $\ch(M) = \bigoplus\limits_{i=0}^{d-1} \ch^i(M)$, called the Chow ring of $M$. It is defined by the quotient:
\[ \ch(M) := \mathbb{Q}[x_F \ |\ F \text{ is a proper flat of }M]/(I+J), \]
where:
\begin{itemize}
    \item $I$ is the ideal generated by the linear relations
    \[ \underset{i \notin F}{\sum} x_F - \underset{j \notin G}{\sum}x_G \]
    for elements $i,j \in [m]$.
    \item $J$ is the ideal generated by the so called incomparability relations
    \[ x_{F}x_{G} \]
    for pairs of incomparable flats $F,G$ of $M$.
\end{itemize}
In the case that the matroid $M$ is realizable over $\Bbbk$ the ring $\ch(M)$ is the Chow ring of a smooth projective variety over $\Bbbk$. This variety can be viewed as the wonderful compactification of the complement of the hyperplanes defined by $\{  p_i \}_{i\in [m]}$ in $\mathbb{P}(V^{\vee})$. Let us recall its construction. Denote by $\mathcal{M}$ the arrangement of hyperplanes corresponding to $M$ and denote by $H_i$ the hyperplane corresponding to $p_i$. For $F \subseteq M$ denote
\[ H_F := \underset{i \in F}{\bigcap} H_i . \]
Then the wonderful model of $\mathcal{M}$, denoted by $\underline{X}_{M}$, is obtained from $\mathbb{P}(V^{\vee})$ by first blowing up $H_F$ for all corank $1$ flats $F$, then for all corank $2$ flats, etc. The resulting variety clearly contains an open subset given by the complement of all $H_i$ in $\mathbb{P}(V^{\vee})$. We will use the description above to give a geometric interpretation of the decomposition in Theorem \ref{thm1} in the realizable case. We slightly abuse the notation by using the matroid in the lower index to denote the wonderful model, as this will simplify our notation when considering restrictions and localizations of arrangements. \\
In general, $\ch(M)$ is the Chow ring of a toric variety $X(\underline{\Sigma}_M)$ associated to the so called Bergman fan $\underline{\Sigma}_M$ of $M$ (see Definition \ref{defbergman}).\\
There exists a degree map
\[ \underline{\mathrm{deg}}_{M}: \ch(M) \rightarrow \mathbb{Q} \]
that is zero on $\ch^{<d-1}(M)$. The degree map is defined by taking the value $1$ on maximal flags
\[ x_{F_1}\cdot \ldots \cdot x_{F_{d-1}} \]
of nonempty proper flats of $M$.
It is a remarkable result of \cite{AHKcomb} that even in the non-realizable case the Chow ring $\ch(M)$ carries the so-called Kähler package, mimicking the properties of a Chow ring of a smooth projective variety.
\begin{prop}
Fix an element $\underline{l}$ of $\ch^1(M)$ corresponding to a strictly convex piecewise linear function on the Bergman fan $\underline{\Sigma}_M$ of $M$
    \begin{enumerate}
        \item (Poincaré duality) Let $k < \frac{d}{2}$. There exists a non-degenerate bilinear pairing
        \[ (-,-)_{\ch(M)}:\ch^k(M) \times \ch^{d-k-1}(M) \rightarrow \mathbb{Q} \]
        given by the formula
        \[ (a,b)_{\ch(M)} = \underline{\mathrm{deg}}_{M}(a\cdot b). \]
        \item (Hard Lefschetz) Let $k < \frac{d}{2}$. Then the map
        \[ \ch^k(M) \rightarrow \ch^{d-k-1}(M) \ \text{ given by } a \mapsto \underline{l}^{d-2k-1}\cdot a \]
        is an isomorphism.
        \item (Hodge-Riemann relations) Let $k < \frac{d}{2}$. Then the bilinear form
        \[ \ch^k(M) \times \ch^k(M) \rightarrow \mathbb{Q}  \ \text{ given by } (a_1,a_2) \mapsto (-1)^k\underline{\mathrm{deg}}_{M}(\underline{l}^{d-2k-1}a_1a_2) \]
        is positive definite on the kernel of multiplication by $\underline{l}^{d-2k}$.
    \end{enumerate}
\end{prop}
Those results led to the proof of the Heron-Rota-Welsh conjecture concerning the log-concavity of the coefficients of the reduced characteristic polynomial of a matroid. In the realizable case over a field of characteristic zero the properties above follow from classical Hodge theory of the corresponding wonderful model \cite{Huhloggraph}.\\
In \cite{ssd} the authors introduced the \textit{augmented Chow ring} $\ach(M)$ of a matroid $M$, a commutative graded algebra $\ach(M) = \bigoplus\limits_{i=0}^{d} \ach^i(M)$ that interpolates between the Chow ring of $M$ and the Möbius algebra of $M$. It is defined by
\[ \ach(M) := \left( \mathbb{Q}[x_F |\ F \text{ is a proper flat of }M ] \otimes \mathbb{Q}[y_i | \ i \in [m]]\right)/(I+J), \]
where:
\begin{itemize}
    \item $I$ is the ideal generated by the linear relations
    \[ y_i - \underset{i\notin F}{\sum} x_F \]
    for any element $i \in [m]$, where the sum runs over all flats of $F$ not containing $i$.
    \item $J$ is the ideal generated by the so called incomparability relations
    \[ x_Fx_G \ \text{ for incomparable flats }F,G \text{ of }M \]
    and
    \[ y_ix_F \ \text{ for } i\notin F. \]
\end{itemize}
Similarly to the classical case, the augmented Chow ring of a realizable matroid is a Chow ring of the smooth projective variety $X_{\mathcal{M}}$. We have a surjective map
\[ \ach(M) \twoheadrightarrow \ch(M) \]
given by sending $y_i \mapsto 0$ for all $i\in [m]$. There exists a unimodular fan called the augmented Bergman complex $\Sigma_M$ (see Definition \ref{defbergman}) such that $\ach(M)$ is the Chow ring of the associated toric variety. We also have a degree map
\[ \mathrm{deg}_M: \ach(M) \rightarrow \mathbb{Q}, \]
which is zero on $\ach^{<d}(M)$ and takes value $1$ on monomials corresponding to maximal cones in the augmented Bergman fan $\Sigma_M$ (see Definition \ref{defbergman}), that is
\[ \left( \prod\limits_{i=1}^ky_i\right)\cdot x_{F_1}\cdot \ldots \cdot x_{F_{d-k}},  \]
 where $F_1\subsetneq \ldots \subsetneq F_{d-k}$ form a flag of proper flats in $M$ and $y_1,\ldots,y_k \in F_1$ form an independent set. The augmented Chow ring also carries a Kähler package.
\begin{prop}\cite[Theorem 1.6]{ssd}
Fix an element $l$ of $\ach^1(M)$ corresponding to a strictly convex piecewise linear function on the augmented Bergman fan $\Sigma_M$ of $M$. Then the following holds:
    \begin{enumerate}
        \item (Poincaré duality) Let $k \leq \frac{d}{2}$. There exists a non-degenerate bilinear pairing
        \[ (-,-)_{\ach(M)}:\ach^k(M) \times \ach^{d-k}(M) \rightarrow \mathbb{Q} \]
        given by the formula
        \[ (a,b)_{\ach(M)} = \mathrm{deg}_{M}(a\cdot b). \]
        \item (Hard Lefschetz) Let $k \leq \frac{d}{2}$. Then the map
        \[ \ach^k(M) \rightarrow \ach^{d-k}(M) \ \text{ given by } a \mapsto l^{d-2k}\cdot a \]
        is an isomorphism.
        \item (Hodge-Riemann relations) Let $k \leq \frac{d}{2}$. Then the bilinear form
        \[ \ach^k(M) \times \ach^k(M) \rightarrow \mathbb{Q}  \ \text{ given by } (a_1,a_2) \mapsto (-1)^k\mathrm{deg}_{M}(l^{d-2k}a_1a_2) \]
        is positive definite on the kernel of multiplication by $l^{d-2k+1}$.
    \end{enumerate}
\end{prop}

\subsection{Pullback and pushforward maps}

We first introduce the Bergman fan of the matroid $M$. For a subset $S$ of $[m]$ (e.g. a flat of $M$ viewed as a subset of the ground set) we denote by $e_S$ the vector $\underset{i\in S}{\sum}e_i \in \mathbb{R}^{[m]}$. We denote by $\rho_{S}$ (respectively $\underline{\rho}_F$) the ray in $\mathbb{R}^{[m]}$ spanned by $e_S$ (respectively the ray in $\mathbb{R}^{[m]}/e_{M}$ spanned by the image of $e_S$).
\begin{definition}{\cite[Definition 2.4]{ssd}} \label{defbergman}
\begin{enumerate}
    \item The Bergman fan $\underline{\Sigma}_M$ of $M$ is a fan in $\mathbb{R}^{[m]}/e_M$, whose cones are given by
    \[ \sigma_{\mathcal{F}} := 
    \mathrm{cone}\{-e_{E\backslash F}\}_{F \in \mathcal{F}}\]
    for $\mathcal{F}$ a flag of nonempty proper flags in $M$. Its rays are $\rho_F$ for flats $F$.
    \item The augmented Bergman fan $\Sigma_M$ of $M$ is a fan in $\mathbb{R}^{[m]}$, whose cones are given by
    \[ \sigma_{I \leq \mathcal{F}} := \mathrm{cone}\{ -e_{E\backslash F} \}_{F\in \mathcal{F}} + \mathrm{cone}\{ e_i \}_{i \in I}  \]
    for $\mathcal{F}$ a flag of proper flats in $M$ and $I$ an independent set in $M$ such that $I \subseteq F$ for every $F \in \mathcal{F}$ (note that this condition is vacuous if $\mathcal{F}$ is the empty flag). Its rays are $\rho_F$ for a flat $F$ and $e_i$ for $i \in [m]$.
\end{enumerate}
\end{definition}
As mentioned in the introduction, the Chow ring (respectively the augmented Chow ring) of a matroid is isomorphic to the Chow ring of the toric variety associated to $\underline{\Sigma}_M$ (respectively to $\Sigma_M$), that is
\[ \ch(M) \simeq \ach(\underline{\Sigma}_M) \ \text{ and } \ \ach(M) \simeq \ach(\Sigma_M).  \]
See \cite[Theorem 3]{Chowringintroduced} or \cite[Section 5.3]{AHKcomb} for the first fact and \cite[Proposition 2.12]{ssd} for the second. We use those isomorphisms to define morphisms of Chow rings by toric geometry. We now describe the stars of the rays $\rho_F$. Recall that the star of a cone $\tau$ of $\Sigma$ is defined as the subset of cones
\[ \mathrm{star}_{\Sigma}(\tau) := \{ \sigma \in \Sigma \ | \ \tau \text{ is a face of } \sigma \}. \]
\begin{prop}{\cite[Proposition 2.7]{ssd}}
    \begin{enumerate}
        \item The star of the ray $\underline{\rho}_F$ in $\underline{\Sigma}_M$ is isomorphic to
        \[ \mathrm{star}_{\underline{\Sigma}_M}(\underline{\rho}_F) \simeq \underline{\Sigma}_{M^F} \times \underline{\Sigma}_{M_F}. \]
        \item The star of the ray $\rho_F$ in $\Sigma_M$ is isomorphic to
        \[ \mathrm{star}_{\Sigma}(\rho_F) \simeq \underline{\Sigma}_{M_F} \times \Sigma_{M^F}. \]
    \end{enumerate}
\end{prop}
\begin{definition}
    We define
    \[ \underline{\psi}_M^F: \ch(M_F) \otimes \ch(M^F) \rightarrow \ch(M) \]
    and
    \[ \underline{\varphi}_M^F: \ch(M) \rightarrow \ch(M_F) \otimes \ch(M^F) \]
    to be the pushforward and pullback homomorphism between Chow rings induced by the morphism of toric varieties defined by the inclusion of fans
    \[ \mathrm{star}_{\underline{\Sigma}_M}(\underline{\rho}_F) \hookrightarrow \underline{\Sigma}_M. \]
\end{definition}
By the projection formula the composition $\underline{\psi}_M^F \circ \underline{\varphi}_M^F$ is multiplication by $x_F$.
\begin{prop}{\cite[Proposition 2.25]{ssd}} \label{formulapsichow}
    Let $S_1$ be any collection of proper flats of $M$ strictly containing $F$ and $S_2$ any collection of nonempty flats of $M$ strictly contained in $F$. Then the pushforward $\underline{\psi}_M^F$ satisfies
    \[ \underline{\psi}_M^F \left( \underset{F_1 \in S_1}{\prod} x_{F_1\backslash F} \otimes \underset{F_2 \in S_2}{\prod} x_{F_2} \right) = x_F \cdot \underset{F_1 \in S_1}{\prod} x_{F_1} \cdot \underset{F_2 \in S_2}{\prod} x_{F_2}. \]
\end{prop}
\begin{definition}
    We define
    \[ \psi_M^F: \ch(M_F) \otimes \ach(M^F) \rightarrow \ach(M) \]
    and
    \[ \underline{\varphi}_M^F: \ach(M) \rightarrow \ch(M_F) \otimes \ach(M^F) \]
    to be the pushforward and pullback homomorphisms between Chow rings induced by the morphism of toric varieties defined by the inclusion of fans
    \[ \mathrm{star}_{\Sigma_M}(\rho_F) \hookrightarrow \Sigma_M. \]
\end{definition}
By the projection formula the composition $\psi_M^F \circ \varphi_M^F$ is multiplication by $x_F$.
\begin{prop}{\cite[Proposition 2.21]{ssd}}
    Let $S_1$ be any collection of proper flats of $M$ strictly containing $F$ and $S_2$ any collection of flats of $M$ strictly contained in $F$. Then the pushforward $\psi_M^F$ satisfies
    \[ \psi_M^F \left( \underset{F_1 \in S_1}{\prod} x_{F_1\backslash F} \otimes \underset{F_2 \in S_2}{\prod} x_{F_2} \right) = x_F \cdot \underset{F_1 \in S_1}{\prod} x_{F_1} \cdot \underset{F_2 \in S_2}{\prod} x_{F_2}. \]
\end{prop}

\subsection{Chow functions of posets}\label{subsecposet}

In this section we recall the definitions of Chow polynomials of partially ordered sets introduced in \cite{CF25}. Their properties play a key role in our proof of the decompositions Theorem \ref{thm1}, Theorem \ref{thm2}, and Theorem \ref{thmaug}. They are part of the so called KLS-theory, which provides a framework for understanding several invariants, called the Kazhdan-Lusztig-Stanley polynomials (KLS for short), attached to combinatorial objects. Those polynomials often have interesting geometric interpretations, usually in terms of point counts of varieties over finite fields and intersection cohomology. They were first introduced in \cite{KLbruhat} for Bruhat intervals, were they agree with dimensions of the graded pieces of intersection cohomology of Schubert varieties. Later, they were generalized by Stanley in \cite{Stanleykernel} for locally graded posets. For a good exposition on the area, see \cite{KLSProudfoot}.\\
We start by recalling some crucial definitions. For $s,t \in \mathcal{P}$ such that $s\leq t$ we define the interval $[s,t]$ to be
\[[s,t]:= \{ u\in \mathcal{P} \ |\ s \leq u \leq t \}. \]
We denote the set of intervals in $\mathcal{P}$ by
\[ \mathrm{Int}(\mathcal{P}) := \{ [s,t]\ |\ s\leq t \} \]
\begin{definition}\label{defincidence}
    We define the incidence algebra $\mathcal{J}(\mathcal{P})$ of $\mathcal{P}$ as the subalgebra of $\mathrm{Mat_{\mathcal{P}\times\mathcal{P}}}(\mathbb{Z}[x])$ defined by the set of matrices whose only nonzero entries correspond to pairs $(s,t) \in \mathcal{P}\times\mathcal{P}$ such that $s\leq t$.
\end{definition}
Equivalently, $\mathcal{J}(\mathcal{P})$ is an algebra of functions from $\mathrm{Int}(\mathcal{P})$ to $\mathbb{Z}[x]$ with product given by convolution
\[ (a \cdot b)_{s,t} = \underset{s\leq u \leq t}{\sum} a_{s,u} \cdot b_{u,t}. \]
The above is a special case of a more general construction, were we consider functions from $\mathrm{Int}(\mathcal{P})$ to an arbitrary ring $R$, in out case $R = \mathbb{Z}[x]$, in the literature it is often assumed to be a field, see \cite{incidenceoriginal}. We define the zeta function $\zeta$ of the poset to be the element of $\mathcal{J}(\mathcal{P})$ satisfying $\zeta_{s,t} = 1$ for all $s,t$. The theory of Chow polynomials was developed in \cite{CF25} for a partially ordered set $\mathcal{P}$ carrying a so called \textit{weak rank function}.
\begin{definition}{\cite[Section 2]{defweakrank}}\label{defweak}
    We call a poset $\mathcal{P}$ with a function $\mathrm{rk}: \mathrm{Int}(\mathcal{P}) \rightarrow \mathbb{Z}$ a weakly ranked poset if it satisfies the conditions
    \begin{itemize}
        \item If $s< t$, then $\mathrm{rk}(s,t) >0$.
        \item If $s\leq u \leq t$, then $\mathrm{rk}(s,t) = \mathrm{rk}(s,u) + \mathrm{rk}(u,t)$.
    \end{itemize}
\end{definition}
In particular for all $s\in \mathcal{P}$ we have $\mathrm{rk}(s,s)=0$. In the rest of the article we assume all posets to be weakly ranked. Given a weakly ranked poset $\mathcal{P}$ we can consider a subalgebra $\mathcal{J}_{\mathrm{rk}}(\mathcal{P}) \subset \mathcal{J}(\mathcal{P})$ defined by
\[ \mathcal{J}_{\mathrm{rk}}(\mathcal{P}) := \{ a\in \mathcal{J}(\mathcal{P}) \ |\ \mathrm{deg}(a_{s,t}) \leq \mathrm{rk}(s,t) \}. \]
\begin{definition}
    We define an operation
    \[ (-)^{\mathrm{rev}}:  \mathcal{J}_{\mathrm{rk}}(\mathcal{P}) \rightarrow \mathcal{J}_{\mathrm{rk}}(\mathcal{P})\]
    such that the entry  $(s,t)$ of $(a)^{\mathrm{rev}}$ is equal to $x^{_{\mathrm{rk}(s,t)}} \cdot (a)_{s,t}(x^{-1})$.
\end{definition}
\begin{definition}
    We say that $\kappa \in \mathcal{J}_{\mathrm{rk}}(\mathcal{P})$ is a kernel if it satisfies the equation
    \[ \kappa^{-1} = \kappa^{rev}. \]
\end{definition}
It is not difficult to observe that the definition of the kernel implies that $x-1 \mid (\kappa)_{s,t}$ for $s<t$ and $(\kappa)_{s,s} = 1$. This leads to the following definition.
\begin{definition}
    We define the reduced kernel $\overline{\kappa}$ as the element of $\mathcal{J}_{\mathrm{rk}}(\mathcal{P})$ given by:
    \[ (\overline{\kappa})_{s,t} := \begin{cases}
        \frac{\kappa_{s,t}}{x-1}\  \text{ if s<t} \\ -1 \ \ \ \text{ otherwise}
    \end{cases} \]
\end{definition}
The most important example of a kernel function is the characteristic function of a poset.
\begin{definition}
    We define the characteristic function $\chi$ of a poset to be the product
    \[ \chi := \zeta^{-1} \cdot \zeta^{\mathrm{rev}} \]
    in the incidence algebra of $\mathcal{P}$.
\end{definition}
The characteristic function can be described as $\chi_{s,t} = \underset{s \leq u \leq t}{\sum} \mu_{s,u}x^{\mathrm{rk}(u,t)}$, were $\mu$ is the Möbius function of a poset defined recursively by
\[ \mu_{s,t} := \begin{cases}
    1 \  \text{ if }s=t \\ -\underset{s\leq u<t}{\sum} \mu_{s,u} \text{ if }s \neq t.
\end{cases} \]
\begin{definition}{\cite[Theorem 2.2]{KLSProudfoot}, \cite[Theorem 6.5]{Stanleykernel}}
    We define the right $KLS$ function $f$ of $\kappa$ to be the unique element of $\mathcal{J}_{\mathrm{rk}}(\mathcal{P})$ satisfying
    \begin{itemize}
        \item $f_{s,s} = 1$ for all $s\in \mathcal{P}$,
        \item $\mathrm{deg}(f_{s,t}) \leq \frac{\mathrm{rk}(s,t)}{2}$ for all $s<t$,
        \item $f^{\mathrm{rev}} = \kappa \cdot f$.
    \end{itemize}
    Similarly the left $KLS$ function $g$ of $\kappa$ is the unique element of $\mathcal{J}_{\mathrm{rk}}(\mathcal{P})$ satisfying
    \begin{itemize}
        \item $g_{s,s} = 1$ for all $s\in \mathcal{P}$,
        \item $\mathrm{deg}(g_{s,t}) \leq \frac{\mathrm{rk}(s,t)}{2}$ for all $s<t$,
        \item $g^{\mathrm{rev}} = g \cdot \kappa$.
    \end{itemize}
\end{definition}
For $\mathcal{P}$ the lattice of flats of a matroid this recovers the $KLS$ polynomials of matroids defined in \cite{KLpolyofmatroid}. We are now ready to define our main objects of study in Section \ref{sectionposet} - the (augmented) Chow functions of posets associated to a kernel function.
\begin{definition}{\cite[Definition 3.3 and Definition 3.13]{CF25}}\label{defchow}
    \begin{enumerate}
        \item We define the Chow function of a weakly ranked poset $\mathcal{P}$ associated to $\kappa$ to be an element 
        \[ H = -(\overline{\kappa})^{-1} \]
        of $\mathcal{J}_{\mathrm{rk}}(\mathcal{P})$.
        \item We define the right augmented Chow function of $\mathcal{P}$ associated to $\kappa$ to be
        \[ F = H \cdot f^{rev}. \]
        \item We define the left augmented Chow function of $\mathcal{P}$ associated to $\kappa$ to be
        \[ G = g^{rev} \cdot H. \]
    \end{enumerate}
\end{definition}

If the poset $\mathcal{P}$ has a minimal element $\hat{0}$ and a maximal element $\hat{1}$ - as in the case for the lattice of flats of a matroid - then we call the polynomial $H_{\hat{0},\hat{1}}$ the $\kappa$-Chow polynomial of $\mathcal{P}$. Similarly, we call $F_{\hat{0},\hat{1}}$ (respectively $G_{\hat{0},\hat{1}}$) the right (respectively left) $\kappa$-augmented Chow polynomial of $\mathcal{P}$. The relevance of the theory above to our considerations of Chow rings of matroids lies in the following theorem.

\begin{prop}{\cite[Theorem 4.9 and Theorem 4.11]{CF25}}\label{HGtochow}
    Let $M$ be a loopless matroid and $\mathcal{P}$ the poset of flats of $M$. Then:
    \begin{enumerate}
        \item The $\chi$-Chow polynomial $H_{\emptyset,M}$ of $\mathcal{P}$ is the Hilbert function of the Chow ring of $M$.
        \item The left $\chi$-Chow polynomial $G_{\emptyset,M}$ of $\mathcal{P}$ is the Hilbert function of the augmented Chow ring of $M$.
    \end{enumerate}
\end{prop}
It is straightforward to see that the above implies that for flats $F \subset G$ of $M$ the polynomial $H_{F,G}$ (respectively $G_{F,G}$) coincides with the Hilbert function of the Chow ring (respectively the augmented Chow ring) of $M_F^G$.
\subsection{Notation}

We always assume that $M$ and $N$ are loopless matroids with ground sets $[m]$ and $[n]$ respectively. We use the letters $F,G$ to denote flats in $M,N$ respectively. We assume that all the tensor products considered are over $\mathbb{Q}$ in the case of Chow rings and over $\mathbb{Z}[x]$ when applied to the elements of the incidence algebra (see Definition \ref{defincidence}).\\
We use $\mathcal{P}$ to denote a poset and assume all posets to be weakly ranked, see Definition \ref{defweak}. We often refer to the invariants of the pair $(\mathcal{P},\kappa)$ of a weakly ranked poset and its kernel function, and sometimes omit $\kappa$ when it is implicitly understood.\\
Let $R$ be a ring with a non-degenerate bilinear form given by the Poincaré pairing, in our case $R = \ch(M)$ or $R = \ach(M)$ for a matroid $M$. We will denote the Poincaré pairing on $R$ by
\[ (-,-)_R : R \times R \rightarrow \mathbb{Q}, \]
indicating the ring in the subscript to avoid confusion. Note that for two rings $R_1,R_2$ with non-degenerate bilinear forms the tensor product $R_1 \otimes R_2$ carries an induced non-degenarate bilinear form
\[ (-,-)_{R_1} \otimes  (-,-)_{R_2}. \]
We denote such a bilinear form on the tensor product by
\[ (-,-)_{R_1\otimes R_2}: R_1\otimes R_2 \times R_1\otimes R_2 \rightarrow \mathbb{Q}, \]
again indicating the ring carrying the form in the subscript.

\section{Construction of the maps and proof of injectivity} \label{sectioninj}

In this section we define the morphisms from the direct sums in Theorem \ref{thm1}, Theorem \ref{thm2}, and Theorem \ref{thmaug} into $\ch(M \oplus N)$ and $\ach(M\oplus N)$ respectively and prove their injectivity. We give a detailed proof of the statements needed for Theorem \ref{thm1}, while for Theorem \ref{thm2} and Theorem \ref{thmaug} we restrict ourselves to defining the maps, as the steps in the proof are analogous. First, we define the structure of a $\ch(M) \otimes \ch(N)$-module on $\ch(M \oplus N)$.

\begin{definition}
    Given two matroids $M,N$ we define the inclusion of rings
    \[ \iota_{M,N}:  \ch(M) \otimes \ch(N) \hookrightarrow \ch(M \oplus N) \]
    by sending
    \[ x_F \otimes 1 \mapsto \underset{G'\subseteq N}{\sum} x_{F \cup G'} \text{ and } 1\otimes x_G \mapsto \underset{F' \subseteq M}{\sum} x_{F'\cup G}, \]
    where we sum over all (possibly empty) flats in $N$ and $M$, respectively.
\end{definition}

A direct calculation on maximal flags shows that it is injective. For realizable matroids this map can be understood geometrically as the map induced by the morphisms $X(\Sigma_{M\oplus N}) \rightarrow X(\Sigma_M)$ and $X(\Sigma_{M\oplus N}) \rightarrow X(\Sigma_N)$ of smooth varieties. In general, it is the morphism of Chow rings induced by the projections of Bergman fans given by the canonical projections
\[ \pi_{M}: \mathbb{R}^{[m] \oplus [n]}/e_{M\cup N} \twoheadrightarrow \mathbb{R}^{[m]}/e_M  \text{ and } \pi_N: \mathbb{R}^{[m] \oplus [n]}/e_{M\cup N} \twoheadrightarrow \mathbb{R}^{[n]}/e_N \]
onto the ground sets of $M,N$.

\begin{definition}
    For nonempty flats $F,G$ of $M,N$ such that $F \cup G \neq M \cup N$ there is a structure of an $\ch(M)\otimes \ch(N)$-module on $\ch(M_F \oplus N_G) \otimes \ch(M^F) \otimes \ch(N^G)$ defined by the composition of maps of rings
    \[ \underline{\varphi}^{F}_M \otimes \underline{\varphi}^G_M : \ch(M)\otimes \ch(N) \rightarrow \ch(M^F) \otimes \ch(M_F) \otimes \ch(N^F) \otimes \ch(N_F), \]
    the obvious map $\sigma$ permuting the factors of the tensor product, and the map of rings
    \[ \iota_{M_F,N_G} \otimes \mathrm{id} \otimes \mathrm{id}: \ch(M_F) \otimes \ch(N_G) \otimes \ch(M^F) \otimes \ch(N^G) \rightarrow \ch(M_F \oplus N_G) \otimes \ch(M^F) \otimes \ch(N^G).\]
\end{definition}
We now define inclusions $u_{F,G}$ of $\ch(M) \otimes \ch(N)$-modules such that the isomorphism in Theorem \ref{thm1} is given by the direct sum of maps
\[ \iota_{M,N} \oplus \underset{\emptyset \subsetneq F\subseteq M, \emptyset \subsetneq N \subseteq M}{\bigoplus} u_{F,G}. \]

\begin{definition}
    For $F,G$ nonempty flats in $M,N$ respectively, such that $F\cup G \neq M \cup N$ we define the map
    \[ u_{F,G}: \ch(M_F \oplus N_G) \otimes \ch(M^F) \otimes \ch(N^G) [-1] \rightarrow \ch(M\oplus N) \]
    by the composition of
    \[ \mathrm{id}\otimes \iota_{M^F,N^G}:  \ch(M_F \oplus N_G) \otimes \ch(M^F) \otimes \ch(N^G) [-1] \rightarrow  \ch(M_F \oplus N_G) \otimes \ch(M^F \oplus N^G) [-1] \]
    and the pushfoward map
    \[ \underline{\psi}^{F \cup G}_{M\oplus N}:  \ch(M_F \oplus N_G) \otimes \ch(M^F \oplus N^G) [-1] \rightarrow \ch(M\oplus N).\]
\end{definition}
Checking that this is a map of $\ch(M) \otimes \ch(N)$-modules amounts to the identity:
\[ x_{F\cup G} \cdot \iota_{M,N} = \underline{\psi}^{F,G}_{M\oplus N} \circ (\iota_{M_F,N_G} \otimes \iota_{M^F,N^G}) \circ \sigma \circ (\underline{\varphi}^F_M \otimes \underline{\varphi}^G_N), \]
which can be verified on monomials using the incomparability relations in the definition of the Chow ring. The only nontrivial part is verifying the equation on the elements of the form $x_F \otimes 1$ and $1 \otimes x_G$, which we do here.
%For a matroid $M$ and $i \in M$ we define elements
%\[ \alpha_M := \underset{i\in F}{\sum} x_F \ \text{ and }\ \beta_M := \underset{i\notin F}{\sum} x_F. \]
%They are independent of the choice of $i \in M$ and by \cite[Proposition 2.24]{ssd} we have $\underline{\varphi}_M^F(x_F) = -1 \otimes \underline{\alpha}_{M^F} - \underline{\beta}_{M_F} \otimes 1$. Now fix two elements $i\in F$, $j \in M\backslash F$. Then we have
By \cite[Proposition 2.24]{ssd} we have
\[
(\underline{\varphi}^F_M \otimes \underline{\varphi}^G_N)(x_F \otimes 1) = - 1 \otimes \underset{i\in K \subsetneq F}{\sum} x_K - \underset{j\notin K, F\subsetneq K \subsetneq M}{\sum} x_{K \backslash F} \otimes 1.
\]
Calculating and applying Proposition \ref{formulapsichow}:
%\begin{align*}
%    (\iota_{M_F,N_G} \otimes \iota_{M^F,N^G}) \circ \sigma \circ (\underline{\varphi}^F_M \otimes \underline{\varphi}^G_N)(x_F \otimes 1) =\\ = - 1 \otimes \underset{i\in K \subsetneq F}{\sum}\underset{L\subseteq N^G}{\sum} x_{K\cup L} - \underset{j\notin K, F\subsetneq K \subsetneq M}{\sum}\underset{L \subsetneq N_G}{\sum} x_{K \cup L \backslash F} \otimes 1
%\end{align*}
%\begin{align*}
 %   \underline{\psi}^{F\cup G}_{M\oplus N} \circ (\iota_{M_F,N_G} \otimes \iota_{M^F,N^G}) \circ \sigma \circ (\underline{\varphi}^F_M \otimes \underline{\varphi}^G_N) = \\ = x_{F\cup G} \cdot \left(- \underset{i\in K \subsetneq F}{\sum}\underset{L\subseteq G}{\sum} x_{K\cup L} - \underset{j\notin K, F\subsetneq K \subsetneq M}{\sum}\underset{G\subseteq L \subsetneq N}{\sum} x_{K \cup L}\right)
%\end{align*}
\begin{align*}
    \underline{\psi}_{M\oplus N}^{F\cup G}\circ (\iota_{M_F,N_G} \otimes \iota_{M^F,N^G}) \circ \sigma \circ (\underline{\varphi}^F_M \otimes \underline{\varphi}^G_N)(x_F \otimes 1) =\\
    = \underline{\psi}_{M\oplus N}^{F\cup G}\left( - 1 \otimes \underset{i\in K \subsetneq F}{\sum}\underset{L\subseteq N^G}{\sum} x_{K\cup L} - \underset{j\notin K, F\subsetneq K \subsetneq M}{\sum}\underset{L \subsetneq N_G}{\sum} x_{K \cup L \backslash F} \otimes 1 \right) = \\
    = x_{F\cup G} \cdot \left(- \underset{i\in K \subsetneq F}{\sum}\underset{L\subseteq G}{\sum} x_{K\cup L} - \underset{j\notin K, F\subsetneq K \subsetneq M}{\sum}\underset{G\subseteq L \subsetneq N}{\sum} x_{K \cup L}\right).
\end{align*}
This is equal to:
\[ x_{F\cup G} \cdot \left(- \underset{i\in K \subsetneq F}{\sum}\underset{\emptyset \subseteq L\subseteq N}{\sum} x_{K\cup L} - \underset{j\notin K, F\subsetneq K \subsetneq M}{\sum}\underset{\emptyset \subseteq L \subsetneq N}{\sum} x_{K \cup L}\right), \]
by the incomparability relations, since all the other terms vanish when multiplied by $x_{F\cup G}$. Since elements containing $j$ and not containing $i$ are automatically incomparable to $F\cup G$, the above is equal to:
\[ x_{F\cup G} \cdot \left(-\underset{i\in H \subseteq M \cup N}{\sum} x_H + \underset{j\in H \subseteq M\cup N}{\sum} x_H + \underset{\emptyset \subseteq L \subseteq N}{\sum} x_{F \cup L} \right). \]
Since for any matroid the sum $\underset{i\in H}{\sum}x_H$ is independent of $i$, the first two terms in the brackets cancel out and the above is equal to
\[ x_{F \cup G} \cdot \left( \underset{\emptyset \subseteq L \subseteq N}{\sum} x_{F \cup L} \right)  = x_{F \cup G} \cdot \iota_{M,N}(x_F \otimes 1),\]
which we wanted to prove.
\begin{definition}
    We define the map
    \[ u_{M,N}: \ch(M) \otimes \ch(N) [-1] \rightarrow \ch(M \oplus N) \]
    as the pushforward $\underline{\psi}_{M\oplus N}^{M}$.
\end{definition}
The verification that this is a $\ch(M)\otimes \ch(N)$-module map is again straightforward.\\

Note that the map $\iota_{M,N}$ is injective and, by \cite[Proposition 2.27]{ssd}, $\underline{\psi}_{M\oplus N}^E$ is injective for all proper flats $E \subsetneq M\oplus N$. Thus $u_{F,G}$ is injective for all non-empty flats $F,G$, as it is a composition of injective maps.\\

\begin{lemma}\label{lemortogonalnosc}
    Let $F,G$ be nonempty proper flats in $M,N$ respectively. Then the $\ch(M) \otimes \ch(N)$-submodule $\mathrm{im}(u_{F,G})$ is orthogonal to all the other summands in Theorem \ref{thm1} with respect to the Poincaré pairing on $\ch(M\oplus N)$.
\end{lemma}
\begin{proof}
    The orthogonality to the factor $\ch(M)\otimes \ch(N)$ follows from the fact $\mathrm{im}(u_{F,G})$ is a $\ch(M) \otimes \ch(N)$-submodule of $\ch(M\oplus N)$ with the multiplication induced by the map $\iota_{M,N}$. In particular, the product of any element of $\mathrm{im}(u_{F,G})$ and $\mathrm{im}(\iota_{M,N})$ lies in $\mathrm{im}(u_{F,G})$, which is zero in the top degree of $\ch(M\oplus N)$.\\
    By definition of $u_{F,G}$, each element of $\mathrm{im}(u_{F,G})$ is a multiple of $x_{F\cup G}$. Thus $\mathrm{im}(u_{F,G})$ is orthogonal to $\mathrm{im}(u_{F',G'})$ for any nonempty flats $F',G'$ such that $F' \cup G'$ is incomparable to $F \cup G$, as well as to $\mathrm{im}(u_{M,N})$.\\
    Thus we are left with proving the orthogonality of $\mathrm{im}(u_{F,G})$ and $\mathrm{im}(u_{F',G'})$ for $F\cup G \subsetneq F'\cup G'$. We show that the product of any two elements of those submodules lies in $\mathrm{im}(u_{F,G})$. Observe that $\mathrm{im}(u_{F,G})$ is the $\ch(M)\otimes\ch(N)$-submodule of $\ch(M\oplus N)$ generated by
    \[ u_{F,G}(\ch(M_F\oplus N_G)\otimes \ch(M^F) \otimes \ch(N^G)). \]
    Thus it is the $\ch(M)\otimes\ch(N)$-submodule spanned by the products
    \[ x_{F\cup G}\cdot x_{H_1}\cdot\ldots\cdot x_{H_k} \]
    for flats $H_1,\ldots,H_k$ strictly containing $F \cup G$. Thus the product $\mathrm{im}(u_{F,G})\cdot \mathrm{im}(u_{F',G'})$ is the submodule spanned by
    \[ x_{F\cup G}x_{F'\cup G'}\cdot x_{H_1}\cdot\ldots\cdot x_{H_k}\]
    for flats $H_1,\ldots,H_k$ strictly containing $F \cup G$ or $F'\cup G'$. Thus the desired inclusion follows from $F\cup G \subsetneq F'\cup G'$.
\end{proof}
\begin{lemma}
    Let $F,G$ be nonempty flats in $M,N$ respectively such that exactly one of the equalities $F=M,G=N$ holds. Then the summand $\ch(M_F \oplus N_G) \otimes \ch(M^F) \otimes \ch(N^G)[-1]$ is orthogonal to all the other summands in Theorem \ref{thm1} except for $\ch(M) \otimes \ch(N) [-1]$.
\end{lemma}
\begin{proof}
    The proof is exactly the same as the proof of Lemma \ref{lemortogonalnosc} except that the factor $\ch(M) \otimes \ch(N) [-1]$ is not Poincaré orthogonal to $\mathrm{im}(u_{F,G})$.
\end{proof}
\begin{lemma}\label{zach}
    Let $F,G$ be nonempty flats such that $F \cup G \neq M \cup N$. Then the inclusion
    \[ u_{F,G}:  \ch(M_F \oplus N_G) \otimes \ch(M^F) \otimes \ch(N^G)[-1] \rightarrow \ch(M\oplus N) \]
    preserves the Poincaré pairing on the two Chow rings up to multiplication by $-1$.
\end{lemma}
\begin{proof}
    Recall from \cite[Proposition 2.25]{ssd} that $\underline{\psi}_{M\oplus N}^{F\cup G}$ is given on monomials by
    \[ \underline{\psi}_{M\oplus N}^{F\cup G}(\underset{F_i \cup G_i \in S_1}{\prod} x_{F_i \cup G_i \backslash F\cup G} \otimes \underset{F_i \cup G_i \in S_2}{\prod} x_{F_i \cup G_i}) = x_{F\cup G} \cdot \underset{F_i \cup G_i \in S_1}{\prod} x_{F_i \cup G_i} \cdot \underset{F_i \cup G_i \in S_2}{\prod} x_{F_i \cup G_i}\]
    for multisets $S_1$ of flats containing $F\cup G$ and $S_2$ of flats contained in $F\cup G$. This shows that for $c, d\in \ch(M_F) \otimes \ch(M^F)$ we have
    \[\underline{\psi}^{F\cup G}_{M\oplus N}(c) \cdot \underline{\psi}^{F\cup G}_{M\oplus N}(d) = \underline{\psi}^{F\cup G}_{M\oplus N}(c\cdot d) \cdot \underline{\psi}^{F\cup G}_{M\oplus N}(1). \]
    Since $\iota_{M^F,N^G}$ is a ring homomorphism we can apply the equation above to
    \[ c := \mathrm{id}\otimes \iota_{M^F,N^G}(a) \text{ and } d:= \mathrm{id}\otimes \iota_{M^F,N^G}(b) \]
    to get
    \[u_{F,G}(a) \cdot u_{F,G}(b) = u_{F,G}(a\cdot b) \cdot u_{F,G}(1) = u_{F,G}(a\cdot b) \cdot x_{F,G}. \]
    Thus to check that $u_{F,G}$ preserves the Poincaré pairing up to $-1$ it is enough to check that
    \[ \mathrm{deg}_{\ch(M\oplus N)}(u_{F,G}(-) \cdot x_{F,G}) = -\mathrm{deg}_{\ch(M_F\oplus N_G)\otimes \ch(M^F) \otimes \ch(N^G)}(-).\]
    We verify this statement on the monomials. It is enough to consider those corresponding to maximal flags in the respective Chow rings. Consider
    \[ a:=  \underset{F_i \cup G_i \in S_1}{\prod} x_{F_i \cup G_i \backslash F\cup G} \otimes \underset{F_i \in S_2}{\prod} x_{F_i}  \otimes \underset{G_i\in S_3}{\prod}x_{G_i},\]
    where $S_1$ is a full flag of flats between $F\cup G$ and $M\oplus N$, $S_2$ is a full flag of flats of $M$ contained in $F$, and $S_3$ is a full flag of flats of $N$ contained in $G$. Note that by definition
    \[ \mathrm{deg}_{\ch(M_F\oplus N_G)\otimes \ch(M^F) \otimes \ch(N^G)}(a) = 1. \]\\
    Assume without loss of generality that $G \neq N$. By the arguments in the proof of Lemma \ref{lemortogonalnosc} $u_{F,G}(a) \cdot x_{F'\cup G}$ lies in the image of $u_{F,G}$ for $F \subsetneq F' \subseteq M$ and in the image of $u_{F',G}$ for $\emptyset \subsetneq F' \subsetneq F$. Since those images lie in degrees $\leq \mathrm{rk}(M)+\mathrm{rk}(N)-2$, we have
    \begin{align}\label{degrow1}
        \mathrm{deg}_{\ch(M\oplus N)}(u_{F,G}(a) \cdot x_{F' \cup G}) = 0.
    \end{align}
    Furthermore, as $\mathrm{im}(u_{F,G})$ is a $\ch(M)\otimes \ch(N)$-submodule, we have $u_{F,G}(a)\cdot \iota_{M,N}(x_G) \in \mathrm{im}(u_{F,G})$ and similarly
    \begin{align}\label{degrow2}
        \mathrm{deg}_{\ch(M\oplus N)}(u_{F,G}(a) \cdot \iota_{M,N}(x_G)) = 0.
    \end{align}
    Subtracting Equation \ref{degrow1} from Equation \ref{degrow2} for all $F'$ as above we get
    \[ \mathrm{deg}_{\ch(M\oplus N)}(u_{F,G}(a) \cdot x_{F\cup G}) + \mathrm{deg}_{\ch(M\oplus N)}(u_{F,G}(a) \cdot x_{\emptyset \cup G})  = 0. \]
    Thus it is enough to show that
    \[ \mathrm{deg}_{\ch(M\oplus N)}(u_{F,G}(a) \cdot x_{\emptyset \cup G}) = 1. \]
    First note that by the incomparability relations in $J_{M\oplus N}$ for each $F_i \in S_2$ the product $x_{F_i\cup G'}x_{\emptyset \cup G}=0$ for any $G'\subsetneq G$. Similarly, $x_{F_i\cup G'}x_{F \cup G}$ for any $G\subsetneq G'$. Furthermore, for each $G_i \in S_3$ we have $x_{F'\cup G_i}x_{\emptyset \cup G} = 0$ for any $F' \neq \emptyset$. Thus all but one monomial in the expansion of $u_{F,G}(a) \cdot x_{\emptyset \cup G}$ vanish and we get
    \[ u_{F,G}(a) \cdot x_{\emptyset \cup G} = \underset{G_i \in S_3}{\prod} x_{\emptyset\cup G_i} \cdot  x_{\emptyset \cup G} \cdot \underset{F_i\in S_2}{\prod} x_{F_i \cup G} \cdot x_{F,G} \cdot \underset{F_i\cup G_i \in S_1}{\prod} x_{F_i \cup G_i}. \]
    Since the flats corresponding to the variables in the monomial above form a maximal flag in $M\oplus N$, the degree map is, by definition, $1$ on $u_{F,G}(a) \cdot x_{\emptyset \cup G}$.
\end{proof}

We are now ready to prove the main theorem of this section.

\begin{theorem}\label{phiinjective}
    The homomorphism
    \[ \phi_{M,N}:=\iota_{M,N} \oplus \underset{\emptyset \subsetneq F\subseteq M, \emptyset \subsetneq N \subseteq M}{\bigoplus} u_{F,G}
    \]
    mapping
    \[ \ch(M)\ot \ch(N) \oplus \underset{F,G}{\bigoplus} \ch(M_F \oplus N_G) \ot \ch(M^F) \ot \ch(N^G) [-1]
    \]
    into $\ch(M\oplus N)$ is injective.
\end{theorem}
\begin{proof}
    Assume that there exists an element
    \[ a:= (a_0 , \{a_{F,G}\}_{F,G}) \]
    such that $\phi_{M,N}(a) = 0$, were
    \[ a_0 \in  \ch(M)\ot \ch(N) \text{ and } a_{F,G} \in \ch(M_F \oplus N_G) \ot \ch(M^F) \ot \ch(N^G) \]
    for any nonempty flats $F,G$ in $M,N$ respectively. Fix two nonempty, proper flats $F,G$ in $M,N$ and assume that $a_{F,G} \neq 0$. Then we can pick $b\in \ch(M_F \oplus N_G) \ot \ch(M^F) \ot \ch(N^G)$ such that
    \[ (a_{F,G},b)_{\ch(M_F \oplus N_G) \ot \ch(M^F) \ot \ch(N^G)}=1. \]
    Then by Lemma \ref{zach} we have
    \[ (u_{F,G}(a_{F,G}),u_{F,G}(b))_{\ch(M\oplus N)} = 1.\]
    However, by Lemma \ref{lemortogonalnosc} 
    \[ (\iota_{M,N}(a_0),u_{F,G}(b))_{\ch(M\oplus N)} = 0 \text{ and } (u_{F',G'}(a_{F',G'}),u_{F,G}(b))_{\ch(M\oplus N)} = 0 \]
    for all pairs of nonempty flats $F',G'$ in $M,N$ such that $F'\cup G' \neq F \cup G$. Thus we must have
    \[ (\phi(a),u_{F,G}(b))_{\ch(M\oplus N)} =1, \]
    a contradiction.\\
    Now assume that $a_{M,N} \neq 0$. Then we can multiply $a_{M,N}$ by an element of $b\in \ch(M) \otimes \ch(N)$ such that
    \[\mathrm{deg}_{\ch(M) \otimes \ch(N)}(a_{M,N} \cdot b) = 1. \]
    Then similarly, since all summands other than $\ch(M) \otimes \ch(N) [-1]$ lie in degrees $\leq \mathrm{rk}(M) + \mathrm{rk}(N) -2$, we would have
    \[ \mathrm{deg}_{\ch(M\oplus N)}(\phi(a) \cdot b) = 1, \]
    a contradiction. Assume $a_{M,G} \neq 0$ for some flat $\emptyset \subsetneq G \subsetneq N$ and choose $b\in \ch(N_G) \otimes \ch(M) \otimes \ch(N^G)$ such that
    \[ (a_{M,G}, b)_{\ch(N_G) \otimes \ch(M) \otimes \ch(N^G)} =1.\]
    Recall that by Lemma \ref{lemortogonalnosc} $\mathrm{im}(u_{M,G})$ is orthogonal to all the summands other than $\mathrm{im}(u_{M,N})$ and itself. Then since $a_{M,N}=0$ we can similarly apply Lemma \ref{zach} to get
    \[ (\phi(a),u_{M,G})_{\ch(M\oplus N)} = 1, \]
    a contradiction. An analogous argument gives $a_{F,N} = 0$ for all nonempty, proper flats $F$. Now $a_0 = 0$ follows from the injectivity of $\iota_{M,N}$, so indeed $a=0$, which proves injectivity of $\phi$.
\end{proof}

We will now define the maps which induce the decomposition in Theorem \ref{thm2}. The inclusion of the summand $\ch(M) \otimes \ch(N)$ is again given by $\iota_{M,N}$.

\begin{definition}
    For proper flats $F,G$ of $M,N$ such that $F \cup G \neq \emptyset$ there is a structure of an $\ch(M)\otimes \ch(N)$-module on $\ch(M^F \oplus N^G) \otimes \ch(M_F) \otimes \ch(N_G)$ defined by the composition of maps of rings
    \[ \varphi_{F} \otimes \varphi_G : \ch(M)\otimes \ch(N) \rightarrow \ch(M^F) \otimes \ch(M_F) \otimes \ch(N^F) \otimes \ch(N_F), \]
    the obvious map $\sigma$ permuting the factors of the tensor product, and the map of rings
    \[ \iota_{M^F,N^G} \otimes \mathrm{id} \otimes \mathrm{id}: \ch(M^F) \otimes \ch(N^G) \otimes \ch(M_F) \otimes \ch(N_G) \rightarrow \ch(M^F \oplus N^G) \otimes \ch(M_F) \otimes \ch(N_G).\]
\end{definition}

\begin{definition}
    Let $F,G$ be proper flats of $M,N$ respectively such that $F \cup G \neq \emptyset$. Then we define
    \[ u'_{F,G}: \ch(M^F \oplus N^G) \otimes \ch(M_F) \otimes \ch(N_G) [-1] \rightarrow \ch(M \oplus N) \]
    by the composition of
    \[ \mathrm{id}\otimes \iota_{M_F,N_G}: \ch(M^F \oplus N^G) \otimes \ch(M_F) \otimes \ch(N_G) [-1] \rightarrow \ch(M^F \oplus N^G) \otimes \ch(M_F \oplus N_G) [-1] \]
    and
    \[ \underline{\psi}^{F\cup G}_{M\oplus N}: \ch(M^F \oplus N^G) \otimes \ch(M_F \oplus N_G) [-1] \rightarrow \ch(M\oplus N). \]
\end{definition}

\begin{definition}
    We define the map
    \[ u_{\emptyset,\emptyset} : \ch(M) \otimes \ch(N)[-1] \rightarrow \ch(M \oplus N) \]
    by the pushforward homomorphism $\underline{\psi}_M^{M\oplus N}$.
\end{definition}

Then the isomorphism in Theorem \ref{thm2} is given by the direct sum of maps
\[ \phi' := \iota_{M,N} \oplus \underset{F\subsetneq M,G \subsetneq N}{\bigoplus} u'_{F,G}. \]
Methods analogous to the ones presented above show the following.
\begin{theorem}
    The map $\phi'$ defined above is injective.
\end{theorem}

We now define the maps inducing the decomposition in Theorem \ref{thmaug}.

\begin{definition}
    Given two matroids $M, N$ we define the inclusion of rings
    \[ \tilde{\iota}_{M,N}: \ach(M) \otimes \ach(N) \rightarrow \ach(M\oplus N) \]
    defined by sending
    \[ x_F\otimes 1 \mapsto \underset{G'\subseteq N}{\sum}x_{F \cup G'} \ \text{ and }\ 1\otimes x_G \mapsto \underset{F'\subseteq M}{\sum} x_{F' \cup G}. \]
\end{definition}
It is the morphism of Chow rings induced by the map of augmented Bergman fans given by the canonical projections
\[ \pi_{M}: \mathbb{R}^{[m] \oplus [n]} \twoheadrightarrow \mathbb{R}^{[m]} \text{ and } \pi_N: \mathbb{R}^{[m] \oplus [n]} \twoheadrightarrow \mathbb{R}^{[n]} \]
onto the ground sets of $M,N$. A direct calculation on maximal flags shows that it is injective.

\begin{definition}
    For proper flats $F,G$ of $M,N$ such that $F \cup G \neq \emptyset$ there is a structure of an $\ach(M)\otimes \ach(N)$-module on $\ach(M^F \oplus N^G) \otimes \ch(M_F) \otimes \ch(N_G)$ defined by the composition of maps of rings
    \[ \varphi_{F} \otimes \varphi_G : \ach(M)\otimes \ach(N) \rightarrow \ach(M^F) \otimes \ch(M_F) \otimes \ach(N^F) \otimes \ch(N_F), \]
    the obvious map $\sigma$ permuting the factors of the tensor product, and the map of rings
    \[ \tilde{\iota}_{M^F,N^G} \otimes \mathrm{id} \otimes \mathrm{id}: \ach(M^F) \otimes \ach(N^G) \otimes \ch(M_F) \otimes \ch(N_G) \rightarrow \ach(M^F \oplus N^G) \otimes \ch(M_F) \otimes \ch(N_G).\]
\end{definition}

\begin{definition}
    Let $F,G$ be proper flats of $M,N$ respectively such that $F\cup G \neq \emptyset$. We define an inclusion
    \[ \tilde{u}_{F,G} : \ach(M^F \oplus N^G) \otimes \ch(M_F)\otimes \ch(N_G)  [-1] \rightarrow \ach(M\oplus N) \]
    of graded $\ach(M) \otimes \ach(N)$-modules as the composition of
    \[ \mathrm{id} \otimes \iota_{F,G}: \ach(M^F \oplus N^G) \otimes \ch(M_F)\otimes \ch(N_G) [-1] \rightarrow \ach(M^F \oplus N^G) \otimes \ch(M_F \oplus N_G) [-1] \]
    and
    \[ \psi^{F \oplus G}_{M \oplus N} :  \ach(M^F \oplus N^G) \otimes \ch(M_F \oplus N_G) [-1] \rightarrow \ach(M\oplus N). \]
\end{definition}
Note that for any loopless matroid $M$ there is an inclusion of graded $\ach(M)$-modules
\[ \iota_{\emptyset}^{M}: \ch(M) \rightarrow \ach(M) \]
identifying $\ch(M)$ with the ideal generated by $x_{\emptyset}$.
\begin{definition}
    We define a map
    \[ \tilde{u}_{\emptyset,\emptyset}: \ch(M) \otimes \ch(N) [-1] \rightarrow \ach(M \oplus N) \]
    by the composition of
    \[ \iota_{M,N}[-1]: \ch(M) \otimes \ch(N)[-1] \rightarrow \ch(M \oplus N)[-1] \]
    with the map
    \[ \iota_{\emptyset}^{M\oplus N}:  \ch(M \oplus N) [-1] \rightarrow \ach(M \oplus N). \]
\end{definition}
We can now describe the map into $\ach(M\oplus N)$ which induces the isomorphism in Theorem \ref{thmaug} by taking the direct sum of the maps defined above.
\begin{theorem}\label{augphiinjective}
    The map $\tilde{\phi}$ defined by the direct sum
    \[ \tilde{\phi}:= \tilde{\iota}_{M,N} \oplus \underset{F \subsetneq M, G \subsetneq N}{\bigoplus} \tilde{u}_{F,G} \]
    is injective.
\end{theorem}
\begin{proof}[Sketch of the proof]
    An argument analogous to that in the proof of Lemma \ref{lemortogonalnosc} shows that the summands
    \[ \ach(M^F \oplus N^G) \otimes \ch(M_F) \otimes \ch(N_G) \]
    are pairwise orthogonal for all pairs of proper flats $F,G$ (including $F\cup G = \emptyset$). Furthermore, each summand above is orthogonal to $\ach(M) \otimes \ach(N)$, as it forms a $\ach(M) \otimes \ach(N)$ with the multiplication in $\ach(M\oplus N)$ agreeing with that induced by the inclusion $\tilde{\iota}_{M,N}$. It is not hard to verify that the Poincaré pairing on each factor is preserved up to a constant. This easily yields that an element of the direct sum mapping to zero has to be zero.
\end{proof}

\section{Hilbert functions and Chow functions of posets} \label{sectionposet}

Consider two weakly ranked posets $\mathcal{P}_1, \mathcal{P}_2$ with kernel functions $\kappa_1,\kappa_2$ and weak rank functions $\mathrm{rk}_1$, $\mathrm{rk}_2$. The goal of this section is to find general formulas for the Chow polynomial and the left $\kappa$-augmented Chow polynomial of $\mathcal{P}_1 \times \mathcal{P}_2$ using the respective invariants of $\mathcal{P}_1$ and $\mathcal{P}_2$. They will specialize to desired equalities between Hilbert functions of (augmented) Chow rings when $\mathcal{P}_1$ and $\mathcal{P}_2$ are lattices of flats of matroids $M$ and $N$.
\begin{definition}
    We define the product $\mathcal{P}_1 \times \mathcal{P}_2$ of weakly ranked posets $\mathcal{P}_1, \mathcal{P}_2$ to be a poset with:
    \begin{itemize}
        \item $\mathcal{P}_1 \times \mathcal{P}_2$ as the underlying set,
        \item the order given by
        \[ (s_1,s_2) \leq (t_1,t_2) \iff s_1 \leq t_1 \text{ and } s_2\leq t_2, \]
        \item a weak rank function defined as
        \[ \mathrm{rk}((s_1,s_2),(t_1,t_2)) = \mathrm{rk}_1(s_1,t_1) + \mathrm{rk}_2(s_2,t_2). \]
    \end{itemize}
\end{definition}
First construct a kernel function of $\mathcal{P}_1 \times \mathcal{P}_2$.

\begin{lemma}
    $\kappa_1 \otimes \kappa_2$ is a kernel function of $\mathcal{P}_1 \times \mathcal{P}_2$.
\end{lemma}
\begin{proof}
    It is enough to observe that both the ${(-)}^{rev}$ and $(-)^{-1}$ operations commute with tensor product of matrices. This follows since
    \[ x^{\mathrm{rk}((s_1,s_2),(t_1,t_2))}(a_1 \otimes a_2)_{(s_1,s_2),(t_1,t_2)} (x^{-1}) = x^{\mathrm{rk}(s_1,t_1)}(a_1)_{s_1,t_1}(x^{-1}) \cdot x^{\mathrm{rk}(s_2,t_2)}(a_2)_{s_2,t_2}(x^{-1}). \qedhere \]
\end{proof}
In the following we denote the kernel function $\kappa_1 \otimes \kappa_2$ of $\mathcal{P}_1 \times \mathcal{P}_2$ by $\kappa$.
\begin{lemma}\label{tensorkls}
    \begin{enumerate}
        \item Let $f_1$ and $f_2$ be right KLS functions of $(\mathcal{P}_1, \kappa_1)$ and $(\mathcal{P}_2, \kappa_2)$ respectively. Then $f_1 \otimes f_2$ is a right KLS function of $(\mathcal{P}_1 \times \mathcal{P}_2, \kappa)$.
        \item Similarly, let $g_1$ and $g_2$ be left KLS functions of $(\mathcal{P}_1, \kappa_1)$ and $(\mathcal{P}_2, \kappa_2)$ respectively. Then $g_1 \otimes g_2$ is a left KLS function of $(\mathcal{P}_1 \times \mathcal{P}_2, \kappa)$.
    \end{enumerate}
\end{lemma}
\begin{proof}
    Since the tensor product commutes with $(-)^{\mathrm{rev}}$ we have:
    \[ (f_1 \otimes f_2)^{\mathrm{rev}} = f_1^{\mathrm{rev}} \otimes f_2^{\mathrm{rev}} =  \kappa_1 f_1 \otimes \kappa_2f_2 = \kappa \cdot (f_1 \otimes f_2).\]
    Furthermore the degree of a polynomial is additive under multiplication, so
    \begin{align*} \mathrm{deg}((f_1\otimes f_2)_{(s_1,s_2),(t_1,t_2)}) = \mathrm{deg}((f_1)_{s_1,t_1} \cdot (f_2)_{s_2,t_2}) = \\ \mathrm{deg}((f_1)_{s_1,t_1}) + \mathrm{deg}( (f_2)_{s_2,t_2})
    < \frac{\mathrm{rk}_1(s_1,t_1)}{2} + \frac{\mathrm{rk}_2(s_2,t_2)}{2} = \mathrm{rk}((s_1,s_2),(t_1,t_2)) \end{align*}
    The second part follows analogously.
\end{proof}
Let us now prove the following.
\chowwfun*
\begin{proof}[Proof of Theorem \ref{thmchowfun}]
    Comparing entries $(s_1,t_1),(s_2,t_2)$ for $s_1,t_1\in \mathcal{P}_1$ and $s_2,t_2\in \mathcal{P}_2$, while considering the cases $s_1=t_1$ and $s_2=t_2$ separately, we see that the following relation between the reduced kernels holds:
    \begin{align}\label{kernelrek}
        \overline{\kappa} = (x-1) \cdot (\overline{\kappa_1} \otimes \overline{\kappa_2}) + x \cdot (\overline{\kappa_1} \otimes I) + x\cdot (I \otimes \overline{\kappa_1}) + xI .
    \end{align}
    Multiplying Equation \ref{kernelrek} from the left by $\overline{\kappa}^{-1}$, and from the right by $\overline{\kappa_1}^{-1} \otimes \overline{\kappa_2}^{-1}$ we get
    \[ \overline{\kappa_1}^{-1} \otimes \overline{\kappa_2}^{-1} = (x-1) \cdot \overline{\kappa}^{-1} + x\cdot \overline{\kappa}^{-1} (I \otimes \overline{\kappa_2}^{-1}) + x \cdot \overline{\kappa}^{-1}(\overline{\kappa_1}^{-1}\otimes I) + x\cdot \overline{\kappa}^{-1}(\overline{\kappa_1}^{-1}\otimes \overline{\kappa_2}^{-1}). \]
    Substituting Definition \ref{defchow} we obtain Equation \ref{rowH1}.
    Similarly, multiplying Equation \ref{kernelrek} from the right by $\overline{\kappa}^{-1}$, and from the left by $\overline{\kappa_1}^{-1} \otimes \overline{\kappa_2}^{-1}$ we get
    \[ \overline{\kappa_1}^{-1} \otimes \overline{\kappa_2}^{-1} = (x-1) \cdot \overline{\kappa}^{-1} + x
    \cdot (I \otimes \overline{\kappa_2}^{-1})\overline{\kappa}^{-1} + x\cdot (\overline{\kappa_1}^{-1}\otimes I)\overline{\kappa}^{-1} + x \cdot (\overline{\kappa_1}^{-1}\otimes \overline{\kappa_2}^{-1})\overline{\kappa}^{-1}. \]
    Substituting Definition \ref{defchow} we obtain Equation \ref{rowH2}.
\end{proof}
\begin{theorem}
    Let $F_1,F_2$, and $F$ be the right augmented Chow functions of $\mathcal{P}_1$, $\mathcal{P}_2$, and $\mathcal{P}_1 \times \mathcal{P}_2$ with respect to $\kappa_1$, $\kappa_2$, and $\kappa$ (see Definition \ref{defchow}). Then we have
    \begin{align}\label{fkerrel}
        F = F_1 \otimes F_2 + x(H_1\otimes H_2)\cdot F - x(I \otimes H_2)F - x(H_1\otimes I)F +xF.
    \end{align}
    Similarly, let $G_1,G_2$, and $G$ be the left augmented Chow functions of $\mathcal{P}_1$, $\mathcal{P}_2$, and $\mathcal{P}_1 \times \mathcal{P}_2$ with respect to $\kappa_1$, $\kappa_2$, and $\kappa$. Then we have
    \begin{align}\label{gkerrel}
        G = G_1 \otimes G_2 + xG\cdot (H_1 \otimes H_2) - xG\cdot (I \otimes H_2) - xG \cdot (H_1 \otimes I) + xG.
    \end{align}
\end{theorem}
\begin{proof}
    By Lemma \ref{tensorkls} the right (respectively left) KLS function of $\mathcal{P}_1\times \mathcal{P}_2$ associated to $\kappa$ is the tensor product of right (respectively left) KLS functions of $\mathcal{P}_1$ associated to $\kappa_1$ and $\mathcal{P}_2$ associated to $\kappa_2$. Thus the relation \ref{fkerrel} is obtained from Equation \ref{rowH2} by multiplying from the right by $f_1^{\mathrm{rev}} \otimes f_2^{\mathrm{rev}}$. Similarly, the relation \ref{gkerrel} is obtained from Equation \ref{rowH1} by multiplying from the left by $g_1^{\mathrm{rev}} \otimes g_2^{\mathrm{rev}}$.
\end{proof}
Using Proposition \ref{HGtochow} the relations above translate into similar equalities for Hilbert functions of Chow rings (respectively augmented Chow rings) of matroids. We use this to give proofs of Theorem \ref{thm1}, Theorem \ref{thm2}, and Theorem \ref{thmaug}.

\maithm*

\begin{proof}[Proof of Theorem \ref{thm1}]
    Denote by $H(M)$ the Hilbert function of the Chow ring of a matroid $M$. Since the characteristic function of the lattice of flats of $M\oplus N$ is the tensor product of the characteristic functions of the lattices of flats of $M$ and $N$ we can apply Equation \ref{rowH1} to the respective $\chi$-Chow functions. Comparing the entries in the row $(\emptyset,\emptyset)$ and column $(M,N)$ of Equation \ref{rowH1} and applying part $(1)$ of Proposition \ref{HGtochow} we get the equality:
    \begin{align*}
        H(M \oplus N)
        = H(M) \cdot H(N) + x \cdot \underset{F\subseteq M, G\subseteq N}{\sum} H(M^F \oplus N^G) \cdot H(M_F) \cdot H(N_G) \\
        - x\cdot \underset{G\subseteq N}{\sum} H(M \oplus N^G)\cdot H(N_G) - x \cdot \underset{F \subseteq M}{\sum} H(M^F \oplus N) + x\cdot H(M \oplus N) = &\\
        = H(M) \cdot H(N) + x \cdot \underset{F\subsetneq M, G\subsetneq N}{\sum} H(M^F \oplus N^G) \cdot H(M_F) \cdot H(N_G)
    \end{align*}
    Theorem \ref{phiinjective} asserts that the map $\phi$ is injective, while the above equation shows that its domain and codomain have equal Hilbert functions. Thus $\phi$ must be an isomorphism.
\end{proof}

We give a brief sketch of a geometric interpretation of the decomposition in Theorem \ref{thm1} in the case when $M$ and $N$ are realizable over a field of characteristic zero. Say that $M$, $N$ correspond to arrangements $\mathcal{M}$, $\mathcal{N}$ in $\mathbb{P}(V^{\vee})$, $\mathbb{P}(W^{\vee})$ respectively. Recall the definition of the wonderful model given in Section \ref{subsectionmatroids}. It can be shown that the wonderful model $\underline{X}_{M \oplus N}$ is isomorphic to a variety obtained from a different order of blowups of $\mathbb{P}(V^{\vee}\oplus W^{\vee})$ where we:
\begin{itemize}
    \item First blow up $H_M$.
    \item Then blow up (the strict transforms of) $H_F$ for flats $F\subsetneq M$ and $H_G$ for flats $G \subsetneq N$ in order determined by increasing corank, obtaining a variety which we denote here by $\underline{\tilde{X}}_{M \oplus N}$.
    \item Then blow up (the strict transforms of) $H_{F\cup G}$ for all other flats $F \cup G \subsetneq M\oplus N$ in the same order as in the constriction in Section \ref{subsectionmatroids}.
\end{itemize}
We have a morphism $\pi':\mathrm{Bl}_{H_M}(\mathbb{P}(V^{\vee} \oplus W^{\vee})) \rightarrow \mathbb{P}(V) \times \mathbb{P}(W)$, which forms a $\mathbb{P}^1$-bundle. We can then blowup the (strict transforms of) $H_F \times \mathbb{P}(W^{\vee})$ and $\mathbb{P}(V^{\vee}) \times H_G$ in $\mathbb{P}(V) \times \mathbb{P}(W)$ and their inverse images $H_F$, $H_G$ in $\mathbb{P}(V^{\vee} \oplus W^{\vee})$. Since $\pi'$ is flat we get a $\mathbb{P}^1$-bundle
\[ \underline{\tilde{X}}_{M\oplus N} \rightarrow \underline{X}_{M} \times \underline{X}_{N}. \]
By the projective bundle formula this corresponds to factors
\[ \ch(M) \otimes \ch(N) \text{ and } \ch(M) \otimes \ch(N) [-1] \]
in Theorem \ref{thm1}. We then consider a sequence of blowups $\pi_{F \cup G}$ of the strict transforms of $H_{F \cup G}$ for pairs of nonempty flats $F, G$ satisfying $F \cup G \neq M \cup N$ in the order as in Section \ref{subsectionmatroids}. At each point the subvariety we blow up has a structure of a $\underline{X}_{M^F} \times \underline{X}_{N^F}$-bundle over $\underline{X}_{M_F \oplus N_G}$. Its cohomology can be calculated by checking the assumptions of the Leray-Hirsch theorem (or, in the case we work over $\mathbb{C}$, considering the affine cell decomposition induced by the Białynicki-Birula decomposition). Applying the formula for the cohomology of the blowup \cite[Theorem 9.27]{voisin} shows, that the blowup $\pi_{F\cup G}$ contributes the factor
\[ \ch(M_F \oplus N_G) \otimes \ch(M^F) \otimes \ch(N^G) [-1] \]
to the decomposition. The composition of all the blowups $\pi_{F \cup G}$ yields the desired morphism
\[ \pi: \underline{X}_{M\oplus N} \rightarrow \underline{X}_{M} \times \underline{X}_{N} . \]
\\
For the sake of completeness we show how to use Theorem \ref{thm1} to obtain a decomposition of $\ch(M \oplus N)$ into irreducible $\ch(M)\otimes \ch(N)$-modules.
\corrr*
\begin{proof}[Proof of Corollary \ref{corirred}]
    Apply recursively the decomposition in Theorem \ref{thm1} to the rings $\ch(M_{F_i} \oplus N_{G_i})$, were $F_i$, $G_i$ is the pair of flats indexing the summand in the direct sum
    \[ \underset{F,G}{\bigoplus} \ch(M_F \oplus N_G) \ot \ch(M^F) \ot \ch(N^G) [-1] \]
    we are considering at the $i$-th step. Fixing one indecomposable summand we construct two flags:
    \[ \mathcal{F} = \{\emptyset \subsetneq F_1 \subsetneq \ldots \subsetneq F_k \subseteq M \}, \  \mathcal{G} = \{\emptyset \subsetneq G_1 \subsetneq \ldots \subsetneq G_k \subseteq N \}.\]
    The recursive process terminates if either at least one of the $F_i, G_i$ is equal to $M, N$ respectively, or if we instead chose the first factor in the recurrence step applying Theorem \ref{thm1}, that is $\ch(M_{F_i}) \otimes \ch(N_{G_i})$. In the first case the considered summand is isomorphic to an element of the direct sum in Corollary \ref{corirred} indexed by a pair of flags as above with $F_k=M$ or $G_k=N$, while in the second case to the summand with the pair of flags satisfying $F_k \subsetneq M$ and $G_k \subsetneq N$. Note that the shifts in degree agree with the formula.
\end{proof}
\begin{remark}
    By examining the map $\phi$ described in Section \ref{sectioninj} that induces the decomposition in Theorem \ref{thm1} we get that the irreducible factors in Corollary \ref{corirred} corresponding to
    \[ \mathcal{F} = \{\emptyset \subsetneq F_1 \subsetneq \ldots \subsetneq F_k \subseteq M \}, \  \mathcal{G} = \{\emptyset \subsetneq G_1 \subsetneq \ldots \subsetneq G_k \subseteq N \} \]
    are spanned by the following monomials:
    \begin{itemize}
        \item $x_{F_1\cup G_1} \cdot \ldots \cdot x_{F_k \cup G_k}$ if $F_k \cup G_k \neq M \cup N$,
        \item $x_{F_1 \cup G_1} \cdot \ldots \cdot x_{F_{k-1}\cup G_{k-1}} \cdot x_{M \cup G_{k-1}}$ if $F_k \cup G_k = M \cup N$.
    \end{itemize}
\end{remark}

\maithmdwa*

\begin{proof}[Proof of Theorem \ref{thm2}]
    The proof is exactly analogous to that of Theorem \ref{thm1}, considering Equation \ref{rowH2} instead of Equation \ref{rowH1}. Here we sketch an alternative method of proof, which deduces the result from Theorem \ref{thm1}. Note that if at least one of the matroids $M$, $N$ has rank one, then the two decompositions are exactly the same, as is the case in \cite[Theorem 1.2]{ssd}. Assume that Theorem \ref{thm2} holds for all pairs of matroids $M',N'$ with $\mathrm{rk}(M') + \mathrm{rk}(N') <n$ and consider $M,N$ with $\mathrm{rk}(M') + \mathrm{rk}(N') = n$. It is enough to show that both sides of the equation
    \[ \ch(M\oplus N) \simeq \ch(M)\ot \ch(N) \ \oplus
         \underset{F,G}{\bigoplus} \ch(M^F \oplus N^G) \ot \ch(M_F) \ot \ch(N_G) [-1] \]
    in Theorem \ref{thm2} have the same decomposition into irreducible $\ch(M) \otimes \ch(N)$-modules. We know the decomposition of the left hand side from Corollary \ref{corirred}. Using our inductive assumption we recursively apply the decomposition from $Theorem \ref{thm2}$ to pairs of matroids $M^F, N^G$ appearing on the right hand side of the equation, getting an analogous formula
    \begin{align*} \underset{\substack{ \emptyset \subseteq F_1 \subsetneq \ldots \subsetneq F_k \subsetneq M \\ \emptyset \subseteq G_1 \subsetneq \ldots \subsetneq G_k \subsetneq N}}{\bigoplus}  \ch(M^{F_1}) \otimes \ch(M_{F_1}^{F_2}) \otimes \ldots \otimes \ch(M_{F_k})
    \otimes \ch(N^{G_1}) \otimes\ch(N_{G_1}^{G_2}) \otimes \ldots \otimes \ch(N_{G_k})[-k]. \end{align*}
    The correspondence between isomorphic summands in the direct sum above and in Corollary \ref{corirred} is given by assigning to a pair of flags
    \[ \mathcal{F} = \{\emptyset \subseteq F_1 \subsetneq \ldots \subsetneq F_k \subsetneq M \}, \  \mathcal{G} = \{\emptyset \subseteq G_1 \subsetneq \ldots \subsetneq G_k \subsetneq N \}\]
    a pair of flags $\mathcal{F}', \mathcal{G}'$, were $\mathcal{F}'$ (respectively $\mathcal{G}'$) is obtained from the original flag by discarding $F_1$ and adding $M$ as the last element if and only if $F_1 = \emptyset$ (respectively discarding $G_1$ and adding $N$ if and only if $G_1 = \emptyset$).
\end{proof}
\begin{remark}
    Note that the isomorphism above differs from the one induced by $\phi'$. For example, the summand spanned by the indecomposable factor corresponding to a pair of flags
    \[ \emptyset = F_1 \subsetneq \ldots \subsetneq F_k \subsetneq M,\ \emptyset \subsetneq G_1 \subsetneq \ldots \subsetneq G_k \subsetneq N\]
    is spanned by the monomial
    \[ x_{\emptyset \cup G_1}\cdot x_{F_2\cup G_2} \cdot \ldots \cdot x_{F_k \cup G_k} \]
    in the decomposition induced by recursively applying $\phi'$. However, the factor corresponding to it under the isomorphism in the proof above is spanned by the monomial
    \[ x_{F_2 \cup G_1} \cdot x_{F_3 \cup G_2} \cdot \ldots \cdot x_{F_{k} \cup G_{k-1}} \cdot x_{M \cup G_k}.\]
\end{remark}
%\begin{itemize}
        %\item $\emptyset \subseteq F_1 \subsetneq \ldots \subsetneq F_k \subsetneq M, \  \emptyset \subseteq G_1 \subsetneq \ldots \subsetneq G_k \subsetneq N$ if $F_1,G_1 \neq \emptyset$
        %\item $\emptyset \subseteq F_1 \subsetneq \ldots \subsetneq F_k \subsetneq M, \  \emptyset \subsetneq G_2 \subsetneq \ldots \subsetneq G_k \subsetneq G_{k+1} =  N$ if $F_1 \neq \emptyset$, $G_1 = \emptyset$
        %\item $\emptyset \subseteq F_2 \subsetneq \ldots \subsetneq F_k \subsetneq F_{k+1} = M, \  \emptyset \subseteq G_1 \subsetneq \ldots \subsetneq G_k \subsetneq N$ if $F_1 = \emptyset$, $G_1 \neq \emptyset$
        %\item 
    %\end{itemize}

\thmaugg*
\begin{proof}[Proof of Theorem \ref{thmaug}]
    Denote by $G(M)$ the Hilbert function of the augmented Chow ring of a matroid $M$. As before, since the characteristic function of the lattice of flats of $M\oplus N$ is the tensor product of the characteristic functions of the lattices of flats of $M$ and $N$, we can apply Equation \ref{rowH1} to the respective $\chi$-Chow functions. Comparing the entries in the row $(\emptyset,\emptyset)$ and column $(M,N)$ of Equation \ref{fkerrel} and applying Proposition \ref{HGtochow} we get the relation:
    \begin{align*}
        G(M \oplus N)
        =G(M) \cdot G(N) + x\cdot \underset{F\subseteq M,G \subseteq N}{\sum} G(M^F \oplus N^G) \cdot G(M_F) \cdot G(N_G) \\
        - x\cdot \underset{G \subseteq N}{\sum} G(M \oplus N^G) \cdot G(N_G) - x\cdot \underset{F \subseteq M}{\sum} G(M^F \oplus N) \cdot G(M_F) + x\cdot G(M \oplus N) = \\
        = G(M) \cdot G(N) + x\cdot \underset{F\subsetneq M, G\subsetneq N}{\sum} G(M^F \oplus N^G) \cdot G(M_F) \cdot G(N_G).
    \end{align*}
    Since by Theorem \ref{augphiinjective} the map $\tilde{\phi}$ is injective, and the above shows that its domain and codomain have the same Hilbert functions, it must be an isomorphism.
\end{proof}
\begin{remark}
    The surjectivity of $\phi$, $\phi'$, and $\tilde{\phi}$ in the proofs of Theorem \ref{thm1}, Theorem \ref{thm2}, and Theorem \ref{thmaug} can be also showed via explicit combinatorial considerations of the relations given by the ideals $I,J$ in the definitions of $\ch(M)$ and $\ach(M)$. However, this is considerably less enlightening than the use of the Chow polynomials of posets.
\end{remark}

\bibliographystyle{alpha}
\bibliography{bibsum}
\end{document}